\documentclass[11pt, oneside]{article}   	
\usepackage{geometry}                		
\usepackage{graphicx}				
\usepackage{amssymb}
\usepackage{amsmath}
\usepackage{amsthm}
\usepackage{harpoon}
\usepackage{accents}

\usepackage{tikz}  
\usepackage{float}
\restylefloat{table}
\usepackage{mathrsfs}
\usepackage{verbatim}

\usetikzlibrary{positioning,chains,fit,shapes,calc}  

\newtheorem{theorem}{Theorem}

\newtheorem{definition}{Definition}
\newtheorem{construction}{Construction}
\newtheorem{remark}{Remark}

\newtheorem{property}{Property}
\newtheorem{requirement}{Requirement}

\newcommand{\vect}[1]{\accentset{\rightharpoonup}{#1}}

\newcommand{\ppp}{\[\begin{aligned}}
\newcommand{\ooo}{\end{aligned}\]}

\title{On Generalizations of the Newton-Raphson-Simpson Method}
\author{Mario DeFranco}

\begin{document}
\maketitle
\abstract{We present generalizations of the Newton-Raphson-Simpson method. 
Given a positive integer $m$ and the coefficients of a polynomial $f(z)$ of degree at least $m$, we define an iterative algorithm NRS($m$) that evaluates, in our terminology, a sum of $m$ formal zeros of $f(z)$. We define NRS($m$) as the $m$-dimensional Newton-Raphson-Simpson method applied to a certain vector-valued function associated to $f(z)$, and we prove that NRS(1) is equivalent to 
the Newton-Raphson-Simpson method. We also prove that NRS($m$) evaluates certain $\mathscr{A}$-hypergeometric series defined by Sturmfels \cite{Sturmfels}. In order to define these algorithms, we make use of combinatorial objects which we call trees with negative vertex degree. }

\section{Introduction}
The main purpose of this paper is to define a sequence of iterative algorithms NRS($m)$ which generalize the Newton-Raphson-Simpson method. Here $m$ is any positive integer, and we prove that NRS(1) is equivalent to the Newton-Raphson-Simpson method.

We first review the Newton-Raphson-Simpson method. Let $f(z)\colon \mathbb{C} \rightarrow \mathbb{C}$ be a differentiable function and $c_0 \in \mathbb{C}$. The Newton-Raphson-Simpson method defines a sequence $c_N, N\geq 0$ by
\[
c_{N+1} = c_N - \frac{f(c_N)}{f'(c_N)}.
\] 
Then the limit $\displaystyle \lim_{N \rightarrow \infty} c_N$, if it exists, is a zero of $f(z)$. Depending on $f(z)$  and $c_0$, the limit may or may not exist. See Kollerstrom \cite{Kollerstrom} for information about the Newton-Raphson-Simpson method. 

The recursive construction of the Newton-Raphson-Simpson method stems from geometry: when $f(x)$ is a real-valued function on $\mathbb{R}$, an $x$-intercept on the graph of $f(x)$ is approximated by taking the $x$-intercepts of successive tangent lines to $f(x)$. We construct NRS($m$) not geometrically but  algebraically, by choosing a particular way to sum the terms in a certain series obtained from $f(x)$. We call this series a formal zero. See DeFranco \cite{DeFranco 1} and \cite{DeFranco 2} for our derivation of formal zeros and proofs of the coefficient formulas. Also see discussion below. 

The Newton-Raphson-Simpson method generalizes to $m$ dimensions. Let $\vect{z}= (z_1, \dots, z_m) \in \mathbb{C}^m$. For any differentiable function $g: \mathbb{C}^m \rightarrow \mathbb{C}^m$

\[
\vect{g}(z) = (g_1(\vect{z}), \ldots, g_{m}(\vect{z})),
\]
define is the Jacobian matrix $J_{\vect{g}}(z)$ for $1\leq i,j\leq m$ by
\[
(J_{\vect{g}})_{i,j}(z)  = \frac{\partial g_{i}}{\partial x_{j}}(\vect{z}). 
\]
Now with $c_N \in \mathbb{C}^m$, the $m$-dimensional Newton-Raphson-Simpson method generates the sequence
\[
c_{N+1} = c_N - J_{\vect{g}}(c_N)^{-1}\vect{g}(c_N)
\] 
with a starting point $c_0$.

We introduce some terminology and results now. In order to simplify some convergence issues regarding the algorithms, we will assume throughout the paper that $f(z)\in \mathbb{C}[z]$ is a polynomial in $z$ of degree $d$: 
\[
f(z) = \sum_{k=0}^d a_i z^k.
\]
Fix a positive integer $m \leq d$. The complex numbers $a_0, \ldots, a_d$ will be the inputs to the algorithm NRS($m$), but in order to derive the algorithm, we first view the coefficients $a_k$ as indeterminates. We use these indeterminates to define certain graded rings $R_m$, and we consider a certain element $A_m \in R_m$ which we call a sum of $m$ formal zeros of $f(z)$. We define $A_m$ essentially as a multi-variate generating function for trees with negative vertex degree. 

We first express $A_m$ as the limit 
 \begin{equation} \label{Jm sum}
A_m =\lim_{n \rightarrow \infty} A_m(n).
 \end{equation}
 for certain other elements $A_m(n) \in R_m$ for integer $n \geq 0$. We then define auxiliary elements $A_{i,m}(n) \in R_m$ for $0 \leq i \leq m-1$ 
 and make the column vector $\vect{A}_m(n)\in (R_m)^m$
 \[
\vect{A}_m(n)= (A_{0,m}(n), \ldots, A_{m-1,m}(n)).
 \] 
 For each $n$, we establish a system of $m$ equations in $R_m$:  we define a certain function $\vect{f}_m: (R_m)^m \rightarrow (R_m)^m$ obtained from $f(z)$ and prove the matrix equation is 
 \begin{equation} \label{eq in Rm}
J_{\vect{f}_m}(\vect{A}_m(n)) \vect{A}_m(n+1) = J_{\vect{f}_m}(\vect{A}_m(n)) \vect{A}_m(n) -  \vect{f}_m(\vect{A}_m(n)).
 \end{equation}
Multiplying \eqref{eq in Rm} through by the inverse of the Jacobian, we obtain the recursive equation for the $m$-dimensional Newton-Raphson-Simpson method applied to $\vect{f}_m$. Now this $\vect{f}_m$ naturally is a function from $\mathbb{C}^m$ to $\mathbb{C}^m$ when we view $a_i$ as elements of $\mathbb{C}$ instead of indeterminates. We thus define NRS($m$) as applying the $m$-dimensional Newton-Raphson-Simpson method applied to $\vect{f}_m$ with arbitrary starting point in $\mathbb{C}^m$. 
 
 We prove that when $m=1$, NRS(1) is the 1-dimensional Newton-Raphson-Simpson method, i.e.
 \[
 J_{\vect{f}_1}(z)^{-1}(\vect{f}_1(z)) = \frac{f(z)}{f'(z)}. 
 \]
 
Given any subset $z_1, \ldots, z_m$ of zeros of $f(z)$, we construct an attractor point $\alpha$
\[
\alpha = (\alpha_0, \ldots, \alpha_{m-1}) 
\] 
for NRS($m$), 
where 
\[
\alpha_0 = \sum_{i=1}^m z_i. 
\]


In Section \ref{} we define the elements $A_m$ in terms of combinatorial objects which we call trees with negative vertex degree, which are rooted plane trees with extra information. We note that \cite{Decoste} and \cite{Labelle} have also interpreted the iterations of the Newton-Raphson-Simpson method using kinds of trees.  

In Section \ref{fim} we explicitly list the auxiliary functions for $d=5$, $1 \leq m \leq 5$. 


See \cite{DeFranco 2} where we prove that $A_m-A_{m-1}$ for $m \geq 2$ is a formal zero. 

In the next section we present a high-level description of the $A_m$, including their appearance in \cite{Sturmfels}.

\subsection{Formal zeros}

\begin{definition}
Let $R$ be a ring containing some elements $a_0, a_1, \ldots, a_d$ and let $f(z)$ denote the function
\[
f(z) = \sum_{i=0}^d a_i z^i
\]
from $R$ to $R$. We say that an element $Z \in R$ is a formal zero of $f(z)$ in $R$ if 
\[
f(Z) = 0 \in R.
\]
\end{definition}

We note that in the above definition we may extend $f(z)$ to be a power series in $z$, provided that it is still well-defined as a function on $R$. 

In Section \ref{}, for a positive integer $m\leq d$, we define a ring $R_m$ and certain element $A_m \in R_m$. As discussed above, these elements $A_m$ will be the focus of this paper. However, the elements $A_m$ for $m\geq 2$ themselves are not formal zeros; the elements $A_m-A_{m-1}$ \textit{are} formal zeros (note that below $A_{m-1}$ can be viewed as an element of $R_m$, so $A_m-A_{m-1}$ exists in $R_m$). This is why we refer to $A_m$ as a sum of the $m$ formal zeros: 
\[
A_m = A_1 + (A_2 - A_1) + \ldots + (A_m - A_{m-1}).
\]

There are some different ways to approach the $A_m$. 

One way, for example, is to view 
\[
Z_m =Z_m(a_0, a_1, \ldots, a_{m-2},a_{m+1}, a_{m+2}, \ldots) 
\]
as a function of the independent variables $a_k$ for $k \neq m-1, m$ and to view $a_{m-1}$ and $a_m$ as constants. Then we set 
\[
Z_m = -\frac{a_{m-1}}{a_m} + \sum_{\vect{n}} c(\vect{n})  a^{\vect{n}} 
\]
 where 
 \[
 \vect{n} = (n_0, n_1, \ldots, n_{m-2}, n_{m+1}, n_{m+2}, \ldots)
 \]
 is a sequence of non-negative integers $n_i$, almost all zero; where 
 \[
 a^{\vect{n}} = \prod_{i=0, \, \neq m-1,m}^\infty a_i^{n_i};
 \] 
 and where $c({\vect{n}})$ are some coefficients. We can solve for the $c({\vect{n}})$ by using the set of equations 
 \[
\frac{ \partial^{\vect{n}}f}{(\partial  a)^{\vect{n}} } (Z_m) \big |_{a_i =0, i\neq m,m-1}=0
 \]
for all $\vect{n}$. This method yields a sum for $Z_m$ that is equal to $A_m-A_{m-1}$. 
 
 Another method to obtain the $A_m$ is to consider the limits of functions in $R_m$. For example, if we let 
 \[
 g_m(z) = z-\frac{f(z)}{a_m z^{m-1}}, 
 \]
 then the limit 
 \[
\lim_{n \rightarrow \infty} g_m^n (-\frac{a_{m-1}}{a_m})
 \]
is equal to $A_m - A_{m-1}$. Again we view $a_{m-1}$ and $a_m$ as constants, and we interpret expressions with denominators as geometric series. 

In \cite{Sturmfels}, Sturmfels considers differential equations satisfied by the roots of a polynomial and expresses their solutions using certain $\mathscr{A}$-hypergeometric series. He gives formulas for the coefficients $c(\vect{n})$ and denotes some of these solutions by 
\[
-\left [ \frac{a_{m-1}}{a_m}\right ]+\left [ \frac{a_{m-2}}{a_{m-1}}{}\right ].
\]
In Section \ref{NRS(m)} we prove that 
\[
A_m = -\left [ \frac{a_{m-1}}{a_m}\right ].
\]

We now describe the outline of this paper. In Section 2 we prove that NRS(1) is equivalent to Newton-Raphson-Simpson method. In Section 3 we define NRS($m$) using the trees with negative vertex degree and some functions built from $f(z)$ that we call auxiliary functions. In Section 4 we show how to explicitly compute the auxiliary functions. In Section 5 we apply NRS($m$) to actual polynomials and present numerical tables of the associated quantities. In Section 6 we discuss further work.

\section{The iteration number of a tree} \label{NRS(1)}
For each plane tree, we define what we call its iteration number. We then show how the Newton-Raphson-Simpson method is actually summing trees ordered by this iteration number (Theorem \ref{m=1 yields Newton-Raphson-Simpson}). The NRS($m$) will also sum trees by iteration number, but the trees will have negative vertex degree.    

We recall the definition of rooted plane trees. See Chapter 5 of \cite{Stanley}. 
\begin{definition} 
A rooted plane tree $T$ is a non-empty finite acyclic graph equipped with the following data: 

1. One vertex is marked as the root, denoted by $\mathrm{root}(T)$.

2. For a vertex $v \neq \mathrm{root}(T)$, let $\mathrm{children}(v)$ denote the set of vertices that are adjacent to $v$ and not on the path from $v$ to the root. If $v =\mathrm{root}(T)$, then $\mathrm{children}(v)$ denotes the set of vertices adjacent to $v$. Each set $\mathrm{children}(v)$ is equipped with a total order $\prec_v$. 

A vertex in $\mathrm{children}(v)$ is called a child of $v$, and we let $\mathrm{deg}(v)$ denote the order of the set $\mathrm{children}(v)$. Each child $u$ of $v$ determines another rooted plane tree $T(u)$, where $u$ is the root of $T(u)$, and $T(u)$ inherits its other data from $T$. For $v \in \mathrm{children}(\mathrm{root}(T))$, we call $T(v)$ a root subtree of $T$.

Let $T_0$ denote the rooted plane tree that consists of only its root vertex.

For an integer $i \geq 0$, we let 
\[
d_i(T)
\]
denote the number of vertices $v$ in $T$ such that $\mathrm{deg}(v) = i$. 
\end{definition} 

\begin{remark} \label{plane tree sequence}
The set of rooted plane trees is in bijection with the set of finite sequences of rooted plane trees. A rooted plane tree $T$ with $v = \mathrm{root}(T)$ corresponds to a sequence of rooted plane trees 
\[
(T(u_1),\ldots, T(u_{n})))
\]
where $n=\mathrm{deg}(v)$ and
\[
\mathrm{children}(v) = (u_1, \ldots, u_{n} )
\]
and
\[
u_i \prec_{v} u_{i+1}.
\]
The rooted plane tree $T_0$ corresponds to the empty sequence.  
 
 \end{remark}

\begin{requirement} 
In this paper we will require that $d_1(T)=0$ for all rooted plane trees.  
\end{requirement}

\subsection{The iteration number of a plane tree}
\begin{definition}
Let $T$ be a rooted plane tree. We define a non-negative integer $\mathrm{iteration}(T)$, which we call the iteration number of $T$, and we say that $T$ is of iteration $n$ if $\mathrm{iteration}(T)=n$. If $T=T_0$ consists of a single vertex, then define $\mathrm{iteration}(T)$ to be 0. Otherwise, define $\mathrm{iteration}(T)$ to be $n+1$ if either of the following two conditions holds: 

\vspace{3mm}

\noindent1. Exactly one of $T$'s root subtrees is of iteration $n+1$ and the rest are of iteration at most $n$. 

\vspace{1mm}
\noindent2.  Two or more of $T$'s root subtrees are of iteration $n$ and the rest are of iteration less than $n$. 

\vspace{3mm}

\noindent If $T$ satisfies the second condition, we say that $T$ is final. 

\end{definition}

\section{NRS($m$) and trees with negative vertex degree} \label{NRS(m)}
To define the algorithms, we define generalized \L{}ukasiewicz words (Definition \ref{generalized Luk}) and trees with negative vertex degree (Construction \ref{tree with negative vertex degree}). The trees with negative vertex degree are rooted plane trees equipped with some extra information.

\subsection{Generalized \L{}ukasiewicz words and trees with negative vertex degree}
Traversing a plane tree using the preorder (depth-first order) of its vertices will be a key concept in defining trees with negative vertex degree. See Chapter 5 of \cite{Stanley} for the definition of preorder. The root of a rooted plane tree is the first vertex visited in the preorder, and if a vertex $u$ is visited before a vertex $v$ in the preorder, we say that $u$ precedes $v$ (or is to the left of $v$) or $v$ succeeds $u$ (or is to the right of $u$) in the preorder. We use the same terminology for trees with negative vertex degree. 

Recall that a plane tree is uniquely determined by the preorder sequence of its vertex degrees (we will often abbreviate ``preorder sequence of its vertex degrees" to ``preorder sequence"). For plane trees, this sequence of non-negative integers is also called the \L{}ukasiewicz word for the tree. We recall the defining properties of \L{}ukasiewicz words. 
\begin{definition} 
A \L{}ukasiewicz word $l$ may be defined as a sequence $(l_i)_{i=1}^N$ of integers 
such that
\[
l_i \geq 0, \hspace{1cm } \sum_{i=1}^n (l_i-1) \geq 0,\hspace{1cm } \mathrm{and}\hspace{1cm } \sum_{i=1}^N (l_i-1) =-1
\]
for each $n<N$. (Note that according to our convention each $l_i \neq 1$ as well.)
\end{definition}

We next define generalized \L{}ukasiewicz words.
 \begin{definition} \label{generalized Luk}
Define a generalized \L{}ukasiewicz word $l$ to be a sequence $(l_i)_{i=1}^N$ of integers such that 
 \[
 l_i \neq 1, \hspace{1cm } \sum_{i=1}^n (l_i-1) \geq 0,\hspace{1cm } \mathrm{and}\hspace{1cm } \sum_{i=1}^N (l_i-1) =-1
\]
for each $n<N$. 
Define $\mathrm{minDegree}(l)$ to be the smallest (most negative) integer $l_i$ that occurs in $l$. For $m \geq 1$, define $\mathrm{Luk}_m$ to be the set of all generalized \L{}ukasiewicz words $l$ such that $\mathrm{minDegree}(l) \geq-m+1$. 
  \end{definition}

We next describe the correspondence between generalized \L{}ukasiewicz words and trees with negative vertex degree. 

\begin{construction} \label{tree with negative vertex degree}
\emph{Given a generalized \L{}ukasiewicz word $l = (l_i)_{i=1}^N$, we construct a tree $T$ with negative vertex degree in the following way. We construct a new word $U(l)$ from $l$ by taking each $l_i$ in $l$ with $l_i<0$ and replacing it with a string of 0's of length $|l_i|+1$.  Thus the generalized \L{}ukasiewicz word 
\[
l = (2,4,3,0,-4,4,0,0,-1)
\]
yields 
\[
U(l) = ( 2,4,3,0,0,0,0,0,0,4,0,0,0,0).
\]
By construction $U(l)$ is an ordinary \L{}ukasiewicz word and thus is the preorder sequence for some classical plane tree which we call $U(T)$. Now from $U(T)$ we construct the tree $T$ by assigning certain vertices of degree 0 in $U(T)$ to have negative degree, and by also by marking certain other vertices of degree 0 in $U(T)$ as ``canceled" vertices.}  

\emph{For each $l_i<0$ in $l$, consider the set of $|l_i|+1$ vertices of degree 0 in $U(T)$ that came from this $l_i$ and take that the rightmost vertex $v$ of these vertices in the preorder. We assign $v$ the vertex degree $l_i$ and mark the other $|l_i|$ vertices as ``canceled" by $v$. In $T$, these canceled vertices do not have vertex degree 0 nor do they contribute 
to the number $d_0(T)$ of vertices of degree 0 in $T$. In fact, we say that a canceled vertex does not have any vertex degree, but we do consider it a child and subtree of its parent vertex. We say that the classical plane tree $U(T)$ is the} underlying tree of $T$. \emph{We say that $T$ has the preorder sequence $l$.}  

\emph{See Figure \ref{tree example}. Thus $T$ has 9 vertices (the filled-in circles) and 5 canceled vertices (the empty circles). Note the graphical depiction of a tree $T$ with negative vertex degree as given in Figure \ref{tree example} determines its generalized \L{}ukasiewicz word in the following way. Traverse $T$ as usual in the preorder, recording in a sequence $l$ the non-negative number of children each vertex has, forgetting for now if a vertex is canceled or non-canceled. Then for each consecutive string of $n$ canceled vertices, take the first non-canceled vertex $v$ that succeeds this string in the preorder and in $l$ change the degree of $v$ from 0 to $-n$. Remove from $l$ the 0's that correspond to canceled vertices. The resulting sequence is the generalized \L{}ukasiewicz word for $T$. } 

\begin{figure}\label{tree example}
\centering
\begin{tikzpicture}

\tikzset{dot/.style={inner sep=1pt,circle,draw,fill},
         circ/.style={inner sep=1pt,circle,draw}}
\node[dot](z){}
       child{node[dot]{}  child{node[dot]{}  child{node[dot]{}}   child{node[circ]{}}  child{node[circ]{}}} child{node[circ]{} } child{node[circ]{}  } child{node[dot]{}  }    
     } child[missing] child[missing] child[missing]  child{node[dot]{}   child{node[dot]{} } child{node[dot]{} } child{node[circ]{} } child{node[dot]{} } };

\end{tikzpicture}
\caption{The tree with negative vertex degree with generalized \L{}ukasiewicz word
$( 2,4,3,0,-4,4,0,0,-1)$. An empty circle indicates a canceled vertex.}
\end{figure}

\end{construction}

For $m\geq 1$, we identify the set of all plane trees whose vertex degrees are at least $-m+1$ with $\mathrm{Luk}_m$. 

\begin{definition}Let $T$ be a tree with negative vertex degree with preorder sequence $l = (l_i)_{i=1}^N$.
We define the iteration number $\mathrm{iteration}(T)$ of $T$ to be equal to $\mathrm{iteration}(U(T))$, where $U(T)$ is the underlying tree of $T$. We say that $T$ is final if $U(T)$ is final.
We define $\mathrm{terminal}(T)$ to be the number of consecutive 0's at the right end of $l$.
\end{definition}

\begin{remark} \label{build trees}
\emph{We can construct any tree $T$ with negative vertex degree by specifying a sequence of trees $( T_1, T_2, \ldots, T_k )$, where each $T_i$ is a tree of negative vertex degree, and then appropriately assigning negative degrees to those trees $T_i$ that consist of a single vertex. That is, suppose $T_i$ is a single vertex and we assign it to have degree $-h<0$. Then there must be a subsequence of the form 
\begin{equation} \label{kh sequence}
(T_{i-k+1}, T_{i-k+2}, \ldots, T_{i-1}, T_i )
\end{equation}
where $T_j $ consists of a single vertex for $i-k+2 \leq j <i$, and $\mathrm{terminal}(T_{i-k+1}) \geq h-(k-2)$. This motivates the following definition. }
\end{remark}

\begin{definition}
 For integers $k$ and $h$ with $m-1\geq h \geq k-1 \geq 1$, define a $(h,k)_m$-block to be a sequence 
 \[
 B=( T_1, T_0, T_0,\ldots, T_0 )
 \]
 of trees in  $\mathrm{Luk}_m$ where there are $k-1$ trees $T_0$ after $T_1$, and $\mathrm{terminal}(T_1) \geq h-(k-2)$. Define a $1_m$-block to be a sequence consisting of a single tree 
 \[
 B=(T_1 )
 \]
 where $T_1$ is any tree in $\mathrm{Luk}_m$. We refer to both $(h,k)_m$-blocks and $1_m$-blocks as blocks. We refer to the tree $T_1$ in a $(h,k)_m$-block or a $1_m$-block as the tree of the block. 
  \end{definition}

\begin{remark} \label{block sequence} \emph{We identify a tree in $\mathrm{Luk}_m$ with a sequence
\[
( B_1, B_2, \ldots, B_N)
\]
where $N\geq 0$ and $B_i$ is either a $(h,k)_m$-block or a $1_m$-block. The tree $T_0$ corresponds to the empty sequence (when $N=0$). We compare this identification to that of Remark \ref{plane tree sequence}.} We call this sequence the block sequence of $T$. 
\end{remark}
\subsection{The number of generalized \L{}ukasiewicz words with a given degree sequence}
Let 
\[
(d_0, d_1, d_2,  \ldots )
\]
 be a sequence of non-negative integers such that $d_1=0$; only finitely many of the $d_k$ are non-zero; and 
 \[
 \sum_{k=0}^\infty (k-1)d_k = -1.
 \]
  Then the number of \L{}ukasiewicz words 
\[
l = (l_1, l_2, \ldots, l_N)
\]  
such that the integer $k$ appears $d_k$ times in $l$ is equal to 
\[
\frac{(\displaystyle \sum_{k=0}^\infty d_k)!}{\displaystyle (\sum_{k=0}^\infty d_k) \prod_{k=0}^\infty (d_k)! }. 
\]
Theorem 5.3.10 of \cite{Stanley} proves this statement. We present a corresponding result about generalized \L{}ukasiewicz words. The proof in \cite{Stanley} directly carries over and we present it here in that generality. 

\begin{theorem} \label{Lukm count}

Let 
\[
d = (\ldots, d_{-2}, d_{-1}, d_0, d_1, d_2,  \ldots )
\]
 be a sequence of non-negative integers such that $d_1=0$; only finitely many of the $d_i$ are non-zero; and 
 \[
 \sum_{i=-\infty}^\infty (k-1)d_k = -1.
 \]
  The number of generalized \L{}ukasiewicz words 
\[
l = (l_1, l_2, \ldots, l_N)
\]  
with degree sequence $d$ is
\[
\frac{(\displaystyle \sum_{k=-\infty}^\infty d_k)!}{\displaystyle (\sum_{k=-\infty}^\infty d_k) \prod_{k=-\infty}^\infty (d_k)! }. 
\]
\end{theorem}
\begin{proof} 
Let 
\[
 \sum_{k=-\infty} ^\infty d_k= N.
 \]

Consider the set $\mathcal{A}_d$ of all sequences 
\[
l = (l_1, l_2, \ldots,  l_N)
\]
such that $d_k$ of the $l_i$ equal $k$ and 
\[
 \sum_{i=-\infty}^\infty (i-1)d_i = -1.
 \]
The order of $\mathcal{A}_d$ is thus 
\[
|\mathcal{A}_d|=\frac{(\displaystyle \sum_{k=-\infty}^\infty d_k)!}{\displaystyle  \prod_{k=-\infty}^\infty (d_k)! }.
\]
Let $l \in \mathcal{A}_d$ and let $C(i,l)$ denote the $i$-th conjugate of $l$:  
\[
C(i;l) = (l_{i+1}, l_{i+2}, \ldots, l_{N-1}, l_N, l_1,l_2,\ldots, l_{i-1} )
\]
We claim that these $N$ conjugates are distinct. If $C(i;l) = C(j;l)$ for $j>i$, then that means 
\[
l_{k} = l_{k'}
\]
whenever $k \equiv k' \mod (j-i)$. This implies that $j-i$ divides $N$ and that each $d_k$ is a multiple of $\frac{N}{j-i}$ . By assumption 
\[
\sum_{k=-\infty }^\infty (k-1)d_k  =-1,
\]
so $\frac{N}{j-i}$ divides $1$. But that means $j-i=N$, which is impossible since $1\leq i,j\leq N$. Therefore the $N$ conjugates of $l$ are distinct. 

We claim that exactly one of these conjugates is a generalized \L{}ukasiewicz word. First we show that at least one conjugate is a generalized \L{}ukasiewicz word. Suppose that the negative integer $M$ is an attained lower bound for the partial sums:
\[
\sum_{i=1}^k (l_i -1)\geq M
\]
for all $1\leq k\leq N$ and that
\[
\sum_{i=1}^{k_1} (l_i -1)= M
\]
with $k_1$ minimal (we may assume that $k_1\neq N$, or else $M=-1$ and we are done). Then we claim that the conjugate $w$
\[
w = (l_{k_1+1}, l_{k_1+2},\ldots,l_N, l_1, l_2,..,l_{k_1})
\]
is a generalized \L{}ukasiewicz word. We have
\[
\sum_{i=k_1+1}^k (l_{i}-1) \geq 0
\]
for all $k_1 \leq k \leq N$, or else $M$ would not be a lower bound.

 Now suppose 
\[
\sum_{i=k_1+1}^N (l_{i}-1) +  \sum_{i=1}^k (l_{i}-1)<0 
\]
for some $1 \leq k <k_1$. Since 
\[
\sum_{i=k_1+1}^N (l_{i}-1) = -M-1,
\]
that implies 
\[
\sum_{i=1}^k (l_{i}-1)<M+1,
\]
contradicting the minimality of $k_1$. 
Therefore $w$ is a generalized \L{}ukasiewicz word. 

Now suppose 
\[
w = (w_1, w_2, \ldots, w_N )
\]
is a generalized \L{}ukasiewicz word. 
If some conjugate $w'$ 
\[
w' =(w_j, w_{j+1}, \ldots, w_N, w_1, w_2, \ldots, w_{j-1})
\]
for $j \neq 1$ is also a generalized \L{}ukasiewicz word, then 
\[
\sum_{i=j}^N (w_i - 1) \geq 0
\]
and 
\[
\sum_{i=j}^N (w_i - 1)+ \sum_{i=1}^{j-1} (w_i - 1) = -1.
\]
Therefore 
\[
 \sum_{i=1}^{j-1}( w_i - 1)<0.
\]
But this contradicts the assumption that $w$ is a generalized \L{}ukasiewicz word. Therefore the only conjugate of $w$ that is a generalized \L{}ukasiewicz word is $w$ itself.

Let $\mathcal{L}_d$ denote the set of generalized \L{}ukasiewicz words with degree sequence $d$. Now $\mathcal{L}_d \subset \mathcal{A}_d$, and we have partitioned $\mathcal{A}_d$ into subsets that each have order $N$ such that each subset contains exactly one generalized \L{}ukasiewicz word. 
Thus 
\[
|\mathcal{L}_d| = \frac{|\mathcal{A}_d|}{N}.
\]
This proves the theorem.
 
\end{proof}

\subsection{The ring $R_m$}
Now we proceed to define the ring $R_m$. For $k \geq 0$, let $R_{m,k}$ be the $\mathbb{Q}$-vector space spanned by all monomials of the form 
\begin{equation} \label{rmk monomial}
 \prod_{i=0}^d a_i^{n_i}
\end{equation}
where the $n_i$ are integers such that 
\[
n_{m-2}+n_{m-1} = -k
\]
and the remaining $n_i \geq 0$ such that
\[
\sum_{i=0}^{m-2} n_i +\sum_{i=m+1}^{d} n_i = k. 
\]

Thus an element $r \in R_{m,k}$ is a finite linear combination of monomials of the form \eqref{rmk monomial}. For $r_{k_1} \in R_{m,k_1}$ and $r_{k_2} \in R_{m,k_2}$, then
\[
r_{k_1}r_{k_2} \in R_{m,k_1+k_2}.
\] 
We let $R_m$ be the ring consisting of all elements $r$ of the form  
\begin{equation} \label{r Rm}
r = \sum_{k=0}^\infty r_k
\end{equation}
where $r_k \in R_{m,k}$; and where addition and multiplication in $R_m$ are the usual operations on infinite sums. Note that in the sum \eqref{r Rm} we allow infinitely many of the $r_k$ to be non-zero. 

\begin{definition}Let $T\in \mathrm{Luk}_m$. Define
\[
w_m(T) = \prod_{k=-m+1}^{d-m+1} (-\frac{a_{m+k-1}}{a_m})^{d_k(T)}
\]
We call $w_m(T)$ the $m$-weight of $T$.  
\end{definition}

\subsection{The element $A_m$}
We can now define $A_m \in R_m$. 
\begin{definition}
\[
A_m = \sum_{T \in \mathrm{Luk}_m} w_m(T).
\] 
\end{definition} 
The elements $A_m$ are well-defined elements of $R_m$ because if $w_m(T) \in R_{m,k}$,
then $k$ 
is equal to the number of non-root vertices of $T$ not of degree 0, and there are only finitely many trees $T$ that have $k$ such vertices whose degrees are bounded by $d-m+1$. 

Note that we can also view $A_{m-1}$ as an element of $R_m$, though we will not use that fact in this paper.

We let
\[
\left [ \frac{a_{j-1}}{a_j}  \right] 
\]
denote an $\mathscr{A}$-hypergeometric series: in equation 4.2 of \cite{Sturmfels}, Sturmfels defines $\displaystyle \left [ \frac{a_{j-1}}{a_j}  \right] $ to be the infinite sum 
\begin{equation} \label{sturmfels def}
\left [ \frac{a_{j-1}}{a_j}  \right] = \sum_{i} \frac{(-1)^{i_j}}{i_{j-1}+1} { i_j \choose i_0, i_1, \ldots, i_{j-1}, i_{j+1}, \ldots, i_n } \left(\frac{a_{j-1}}{a_{j}^{i_j+1}}\right) \prod_{k=0, \, k \neq j}^n a_k^{i_k}
\end{equation}
where the sum is over all sequences $i$ of non-negative integers $(i_0, i_1, \ldots, i_n)$ such that
\begin{equation} \label{vertices ik}
\sum_{k=0,\,  k \neq j} ^n i_k = i_j 
\end{equation}
and 
\begin{equation} \label{degree ik}
\sum_{k=0,\,  k \neq j} ^n k i_k = ji_j.
\end{equation}

\begin{theorem} With $a_i=0$ for $i>d$, the $\mathscr{A}$-hypergeometric series $\displaystyle \left [ \frac{a_{m-1}}{a_m}  \right]$ may be viewed as an element of $R_m$. As elements of $R_m$,
\[ 
A_m = -\left [ \frac{a_{m-1}}{a_m}\right ]. 
\]
\end{theorem}

\begin{proof}To agree with the notation of \cite{Sturmfels}, we let $j=m$. 

Using equation \eqref{vertices ik}, equation \eqref{degree ik} may be rewritten as 
\[
-( i_{j-1}+1)+\sum_{k=0, \, k\neq j, \, j-1}^n (k-j)i_k =-1. 
\]
And 
\begin{equation} \label{interpret count}
 \frac{1}{i_{j-1}+1}{ i_j \choose i_0, i_1, \ldots, i_{j-1}, i_{j+1}, \ldots, i_n } =  \frac{1}{i_j+1}{i_j+1\choose i_0, i_1, \ldots, i_{j-2},i_{j-1}+1, i_{j+1}, \ldots, i_n }. 
\end{equation}
Thus we can interpret each $i_k, k\neq j, j-1 $ as the number of vertices in a tree $T$ with negative vertex degree that have degree $1+k-j$; $i_{j-1}+1$ as the number of vertices that have degree $0$; and $i_j$ as the number of vertices that have vertex degree (that is, are not canceled). By Theorem \ref{Lukm count}, expression \eqref{interpret count} counts the number of all such $T$. The monomial factor in \eqref{sturmfels def} is then $-w_j(T)$. This completes the proof.   
\end{proof}

\subsection{Auxiliary functions $f_{i,m}(x)$}

We perform this sum by ordering the trees $T$ according to their iteration number: 
letting 
\begin{equation} \label{J_m(n)}
A_m(n) = \sum_{T \in \mathrm{Luk}_m, \mathrm{iteration}(T) \leq  n} w_m(T),
\end{equation}
and
\begin{equation} \label{sum by iteration}
A_m = \lim_{n \rightarrow \infty} A_m(n).
\end{equation}
This limit makes sense because for a fixed $k$, the component of $A_m(n)$ in $R_{m,k}$ stabilizes for sufficiently large $n$.
 We introduce the quantities $A_{i,m}(n)$ and establish a system of $m$ equations that are linear in the $A_{i,m}(n)$.  To define $A_{i,m}(n)$ we proceed as follows.

First, recall the construction of trees in $\mathrm{Luk}_m$ discussed in Remark \ref{build trees}. Given integers $h$ and $k$, the number $\mathrm{terminal}(T)$ determines whether $T$ is a valid choice for the tree of an $(k,h)_m$-block. Therefore we define subsets of $\mathrm{Luk}_m$ based on $\mathrm{terminal}(T)$: 

\begin{definition} 
For integer $0 \leq i \leq m-1$, define 
\[
\mathrm{Luk}_{i,m} = \{T \in \mathrm{Luk}_m: \mathrm{terminal}(T) \geq i \}.
\] 

\end{definition}

Thus 
\[
\mathrm{Luk}_m = \mathrm{Luk}_{0,m}
\]
 Refining by the iteration number yields the following terms.  
\begin{definition}
For $0\leq i \leq m-1$, let
\[
\mathrm{Luk}_{i,m}(n) = \{T \in \mathrm{Luk}_{i,m}: \mathrm{iteration}(T) \leq n\}.
\] /Users/marioadefranco/Desktop/Math Tex/Class Tex files/NRS iteration number.tex

\begin{equation} \label{J_i,m(n)}
A_{i,m}(n) = \sum_{T \in \mathrm{Luk}_{i,m}(n)} w_m(T)
\end{equation}
\end{definition}


Thus for $n \geq 1$
\[
J_{m}(n) = \sum_{i=0}^{m-1}J_{i,m}(n).
\]
We next explain how to establish the system of $m$ linear equations satisfied by $A_{i,m}(n)$.

 For general $m$, we will use $m$ auxiliary functions 
 \[
 f_{i,m}(x) \colon R_m \rightarrow R_m
 \]
 where $x$ denotes  the $m$-tuple 
 \[
x  = (x_0, x_1, \ldots, x_{m-1}) 
 \]
We construct the auxiliary functions to have Property \ref{auxiliary function property 1} below.

\begin{definition}
Let $X$ be a subset of $\mathrm{Luk}_m$. 
Define the set $\mathrm{Supertrees}_m(X) \subset \mathrm{Luk}_m$ to be the set of trees $T$ such that if $T'$ is the tree of a block of $T$, 
then $T' \in X$.  

 
 For $X \subset \mathrm{Luk}_m$, define the element $w_m(X;i) \in R_m$
\[
w_m(X;i) = \sum_{T \in  X \cap \mathrm{Luk}_{i,m}} w_m(T)
\]
and let $w_m(X)$ denote the $m$-tuple 
\[
w_m(X) = (w_m(X;0), w_m(X;1), \ldots, w_m(X;m-1)).
\]
\end{definition}

\begin{property} \label{auxiliary function property 1}
\[
f_{i,m}(w_m(X) )= \sum_{T \in \mathrm{Supertrees}_m(X) \cap \mathrm{Luk}_{i,m}} w_m(T).
\]

\end{property}
Thus $f_{i,m}(x)$ outputs the $m$-weights of trees in $ \mathrm{Luk}_{i,m}$ and with prescribed trees in their blocks.

The variable $x_i$ is a placeholder for the $m$-weight of any tree in $\mathrm{Luk}_{i,m}$.
  
To construct the $f_{i,m}(x)$, we consider all possible block sequences of trees in $\mathrm{Luk}_{i,m}$: to simplify notation we say that a tree $T$ has block sequence 
\[
(B_N, \dots, B_2, B_1)
\] 
where $B_{i+1}$ precedes $B_i$ in the preorder of $T$. 
We assign an expression to each block type. The tree of an $(h,k)_m$-block must be $\mathrm{Luk}_{h-(k-2),m}$, and $h-(k-2)$ of it terminal vertices are canceled by the vertex of degree $-h$. Therefore we assign to an $(h,k)_m$-block the expression 
\begin{equation} \label{h,k expression}
 \mathrm{expr}(x;h,k)=x_{h-(k-2)} (-\frac{a_{m-1}}{a_m})^{k-2-h} (-\frac{a_{m-1-h}}{a_m})
\end{equation}
The tree of a $1_m$-block can have any number of terminal vertices, so we assign to a $1_m$-block the expression
\[
 \mathrm{expr}(x;1)=x_0.
\]

This motivates the following function $\mathrm{PT}(x,s)$. For an integer $s\geq 0$, we also allow the last $s$ root subtrees of $T$ in the preorder to be ``unspecified", which we will specify afterward depending on which $\mathrm{Luk}_{i,m}$ $T$ is in. We assign to a block sequence with $s$ unspecified root subtrees the product of the block expression times 
\[
-\frac{a_{m-1+i+s}}{a_m}\mathbf{1}(i+s \geq 2)
\] 
which comes from the $m$-weight of the root of $T$. Summing over all possible non-empty sequences of block types gives the function $\mathrm{PT}(x,s)$:
\begin{definition} For positive integer $k$, define 
\[
 \mathrm{expr}_m(x;k) = \begin{cases} 
 0 &\text{ if } m<k \\
\displaystyle \sum_{h=k-1}^{m-1} \mathrm{expr}_m(x;h,k) &\text{ if $1<k \leq m$}\\
  x_0 &\text{ if $k=1$}
  \end{cases}
\]
and 
\[
\mathrm{PT}(x,s)=\sum_{i=0}^{d-m-s+1} -\frac{a_{m-1+i+s}}{a_m}\mathbf{1}(i+s \geq 2)\sum_{c \in C(i)}  \prod_{j=1}^{\mathrm{length}(c)} \mathrm{expr}(x;c(j))
\]
where $C(i)$ is the set of compositions $c$ of $i$ 
\[
c = (c(1), \ldots, c(n))
\]
with positive integer parts and $\mathrm{length}(c)=n$, and for a statement $W$
\[
\mathbf{1}(W) = 
\begin{cases}
1 \text{ if } W \text{ is true } \\ 
0 \text{ otherwise }.  
\end{cases}
\]
\end{definition}
Note that 
\[
\mathrm{PT}(x,0)
\]
is the sum of all expressions arising from all possible non-empty sequences of block types. 

Now we can define $f_{0,m}(x)$: 
\[
f_{0,m}(x) = -\frac{a_{m-1}}{a_m} +\mathrm{PT}(x,0) 
\]
Now we can define $f_{i,m}(x)$ for $1 \leq i \leq m-1$. If $T \in \mathrm{Luk}_{i,m}$, then either $B_1$ is a $1_m$-block whose tree is in $\mathrm{Luk}_{i,m}$, or  for some $1\leq n \leq i-1$
\[
B_1=B_2=\ldots = B_{n}=(T_0)
\]
and 
$B_{n+1}$ is a $1_m$-block whose tree is in $\mathrm{Luk}_{i-n,m}$. If $i=1$, then $T$ may also equal $T_0$. Therefore 
\[
f_{i,m}(x) =-\frac{a_{m-1}}{a_m} \mathbf{1}(i=1)+\sum_{n=0}^{i-1} x_{i-n}(-\frac{a_{m-1}}{a_m})^n
 \mathrm{PT}(x;n+1).
 \]
 
\begin{definition} Let $T_1 \in \mathrm{Luk}_m \setminus X$. Define the set $\mathrm{Supertrees}_m(X, T_1) \subset \mathrm{Luk}_m$ to be the set of trees $T$ such that $T$ has exactly one block whose tree is $T_1$, and the trees of the remaining blocks are in $X$.
 \end{definition} 
 
\begin{property}  \label{auxiliary function property 2}
Let $T_1 \in \mathrm{Luk}_{j,m} \setminus X$.  Then 
\[
w_m(T_1) \frac{\partial f_{i,m}}{\partial x_j}(w_m(X))  =  \sum_{T \in \mathrm{Supertrees}_m(X,T_1) \cap \mathrm{Luk}_{i,m}} w_m(T).
\]
\end{property}

\subsection{The system of linear equations for $A_{i,m}(n)$} 

\begin{definition}
Recall $x$ denotes the $m$-tuple
\[
x = (x_0, \ldots, x_{m-1}).
\]
Define the function 
$\vect{f}_m \colon (R_m)^m \rightarrow (R_m)^m$ by 
\[
\vect{f}_m(x) = (x_0 - f_{0,m}(x), \ldots, x_{m-1}  - f_{m-1,m}(x)).
\]
For any function 
\[
g(x) = (g_0(x), \ldots, g_{m-1}(x)),
\]
define is the Jacobian matrix $J_g(x)$ for $1\leq i,j\leq m$ by
\[
(J_g)_{i,j}(x )  = \frac{\partial g_{i-1}}{\partial x_{j-1}}(x). 
\]
\end{definition} 

\begin{theorem} 
\end{theorem}
\begin{proof}
By Properties \ref{auxiliary function property 1} and \ref{auxiliary function property 2}
\[
 \sum_{T \in \mathrm{Luk}_{i,m}, \text{ iteration($T$) $\leq n$}} w_m(T) + \sum_{T \in \mathrm{Luk}_{i,m}, \text{ iteration($T$) $= n+1$, \text{ $T$ is final}}} w_m(T)=f_{i,m}(A_m(n)) 
\]
and
\[
 \sum_{T \in \mathrm{Luk}_{i,m}, \text{ iteration($T$) $= n+1$, \text{ $T$ is not final}}} w_m(T)=\sum_{j=0}^{m-1} (A_{j,m}(n+1) - A_{j,m}(n))\frac{\partial f_{i,m}(A_m(n))}{\partial x_j}. 
\]
Adding these two equations yields 
\[
A_{i,m}(n+1) = f_{i,m}(A_m(n)) + \sum_{j=0}^{m-1} (A_{j,m}(n+1) - A_{j,m}(n))\frac{\partial f_{i,m}(A_m(n))}{\partial x_j}.
\]
Viewing 
\[
A_m(n)=(A_{0,m}(n), \ldots, A_{m-1,m}(n))
\]
as a column vector, we can re-arrange \eqref{} and take all $m$ equations to obtain the one matrix equation
\[
J_{\vect{f}_m}(A_m(n)) A_m(n+1) = J_{\vect{f}_m}(A_m(n)) A_m(n)  - \vect{f}_m(A_m(n)). 
\] 
This completes the proof.
\end{proof}

Assuming $J_{\vect{f}_m}(A_m(n))$ is invertible, we have 
\[
A_m(n+1) = A_m(n) - J_{\vect{f}_m}(A_m(n))^{-1} \vect{f}_m(A_m(n))
\]
This is the $m$-dimensional Newton-Raphson-Simpson method applied to the function $\vect{f}_m$ with starting point $A_m(0)$ which is 
\[
A_m(0) = (A_{0,m}(0), \ldots, A_{m-1,m}(0))
\]
where 
\[
A_{0,m}(0) = A_{1,m}(0) = -\frac{a_{m-1}}{a_m}
\]
and 
\[
A_{i,m}(0) = 0
\]
for all other $i$. 

\begin{definition} 
For $a_m \neq 0$, we thus define NRS($m$) as applying the $m$-dimensional Newton-Raphson-Simpson to the function $\vect{f}_m: \mathbb{C}^m \rightarrow \mathbb{C}^m$  with arbitrary starting point in $\mathbb{C}^m$. 
\end{definition}

\begin{theorem} \label{m=1 yields Newton-Raphson-Simpson}
For $a_1 \neq 0$, then NRS(1) is the 1-dimensional Newton-Raphson-Simpson method. 
\end{theorem}
\begin{proof} 
For $m=1$, $x$ is the 1-tuple $x_0$. 
We have that $\vect{f}_1(x_0)$ is 
\begin{align*}
x_0 - f_{0,1}(x_0) &= x_0 -(-\frac{a_0}{a_1} +\sum_{i=2}^d -\frac{a_i}{a_1}x_0^i )\\ 
 &= \frac{f(x_0)}{a_1}
\end{align*}
and 
\[
J_{\vect{f}_1}^{-1}(x_0) = \frac{a_1}{f'(x_0)}.
\]
Thus the sequence of iterations is given by 
\begin{align*}
c_{N+1} &= c_N - J_{\vect{f}_1}^{-1}(c_N) \vect{f}_1(c_N)\\ 
& = c_N - \frac{f(c_N)}{f'(c_N)}.
\end{align*}
This completes the proof.
\end{proof}
We will explicitly construct the auxiliary functions and find numerical solutions to these systems in Section \ref{fim}. 

\section{Explicit construction of the auxiliary functions} \label{fim}

We list the auxiliary functions for a quintic polynomial
\[
f(z) =a_0 + a_1z+a_2z^2 +a_3z^3+a_4z^4+a_5z^5.
\]

\noindent $\mathbf{m=1}$: 
\[
f_{0,1}(x_0) = -\frac{a_0}{a_1} -\frac{a_2}{a_1}x_0^2+-\frac{a_3}{a_1}x_0^3+-\frac{a_4}{a_1}x_0^4+-\frac{a_5}{a_1}x_0^5
\]

\noindent $\mathbf{m=2}$: 
\begin{align*}
f_{0,2}(x_0,x_1)&=  -\frac{a_1}{a_2}  - \frac{a_3 x_0^2}{a_2} - \frac{a_4 x_0^3}{a_2} - \frac{
 a_5 x_0^4}{a_2} -  \frac{a_0 a_3 x_1}{a_1 a_2} -  2\frac{
  a_0 a_4 x_0 x_1}{a_1 a_2)} -  3\frac{ a_0 a_5 x_0^2 x_1}{
 a_1 a_2} -  \frac{a_0^2 a_5 x_1^2}{a_1^2 a_2} \\ 
f_{1,2}(x_0,x_1)&=-\frac{a_1}{a_2}  + 
 x_1 (-\frac{a_3 }{a_2}x_0  -\frac{a_4 }{a_2} x_0^2  -\frac{a_5 }{
    a_2}x_0^3  -\frac{a_0 a_4 }{a_1 a_2} x_1 - 2 \frac{a_0 a_5 }{
    a_1 a_2}x_0 x_1)
\end{align*}
\noindent $\mathbf{m=3}$: 
\begin{align*}
f_{0,3}(x_0, x_1,x_2) = & -\frac{a_2}{a_3} - \frac{a_4 x_0^2}{a_3} - \frac{a_5 x_0^3}{a_3} - \frac{
 a_0 a_5 x_1)}{a_2 a_3} - \frac{
 a_4}{a_3} (\frac{a_1 x_1}{a_2} - \frac{a_0 a_3 x_2}{a_2^2}) -
 2  \frac{a_5 x_0}{a_3} ( \frac{a_1 x_1}{a_2} - \frac{a_0 a_3 x_2)}{a_2^2})) \\
f_{1,3}(x_0, x_1,x_2) = & -\frac{a_2}{a_3}  + 
 x_1 (-(\frac{a_4 }{a_3}x_0 - \frac{a_5 }{a_3}x_0^2 
 - \frac{ a_5}{a_3} (\frac{a_1 }{a_2}x_1 - \frac{a_0 a_3 }{a_2^2} x_2) \\ 
f_{2,3}(x_0, x_1,x_2) = & -((\frac{a_2}{a_3} (\frac{(a_4}{a_3} - (\frac{a_5}{a_3}x_0) x_1))  + 
x_2 (-(\frac{a_4 }{a_3}x_0 - \frac{a_5}{a_3} x_0^2 - (
    \frac{a_5}{a_3} (\frac{a_1}{a_2}x_1 - \frac{a_0 a_3}{a_2^2}x_2)))  
\end{align*}

\noindent $\mathbf{m=4}$: 
\begin{align*}
f_{0,4}(x_0,x_1,x_2,x_3) = &-(a_3/a_4) - x_0 - (a_5 x_0^2)/a_4 - (
 a_5 ((a_2 x_1)/a_3 - (a_1 a_4 x_2)/a_3^2 + (
    a_0 a_4^2 x_3)/a_3^3))/a_4\\ 
f_{1,4}(x_0,x_1,x_2,x_3)= &-(a_3/a_4) - x_1 - (a_5 x_0 x_1)/a_4\\ 
f_{2,4}(x_0, x_1,x_2,x_3) = &(a_3 a_5 x_1)/a_4^2 - x_2 - (a_5 x_0 x_2)/a_4\\  
f_{3,4}(x_0,x_1,x_2,x_3) = & (a_3 a_5 x_2)/a_4^2 - x_3 - (a_5 x_0 x_3)/a_4
\end{align*}

\noindent $\mathbf{m=5}$: 
\begin{align*}
f_{0,5}(x_0, x_1, x_2,x_3,x_4)&= -\frac{a_4}{a_5}\\ 
f_{1,5}(x_0, x_1, x_2,x_3,x_4) &= -\frac{a_4}{a_5}\\  
f_{2,5}(x_0, x_1, x_2,x_3,x_4) &=  0 \\ 
f_{3,5}(x_0, x_1, x_2,x_3,x_4) &=  0\\ 
f_{4,5}(x_0, x_1, x_2,x_3,x_4) &=  0\\
\end{align*}

\end{document}